\documentclass[12pt]{amsart}
\usepackage{amsmath,amssymb,mathtools,amsthm,textcomp,mathscinet}
\usepackage{amsfonts,graphicx,enumerate,subcaption}
\usepackage{color,tikz,float}
\usepackage[mathscr]{eucal}
\usepackage[foot]{amsaddr}
\theoremstyle{definition}
\usepackage[american]{babel}
\usepackage{csquotes}
\usepackage{placeins}

\newtheorem{theorem}{\bf Theorem}[section]
\newtheorem{remark}{\bf Remark}[section]

\usepackage{chngcntr} \counterwithin{equation}{section}

\usepackage[colorinlistoftodos]{todonotes}
\begin{document}
\title{Superposition of time-changed Poisson processes and their hitting times}
%\date{}
\author{ A. Maheshwari$^*$,~E. Orsingher$^\#$, and A. S. Sengar$^@$}
\address{\small $^*$Postal address: Operations Management and Quantitative Techniques,
	Indian Institute of Management Indore, Prabandh Shikhar, Rau-Pithampur Road, Indore 453556, Madhya Pradesh, INDIA.\newline 
	$^\#$Sapienza University of Rome 
	P.le Aldo Moro, 5, 00185 Roma, ITALY.\newline
	$^@$Department of Mathematics, Indian Institute of Technology Madras, Chennai 600036, 	INDIA.}

%\address{\small ${}^2$Department of Statistics and Probability, Michigan State University, East Lansing, MI 48824, USA.}
\email{adityam@iimidr.ac.in,enzo.orsingher@uniroma1.it,ma15d201@smail.iitm.ac.in}	

\begin{abstract}
	\noindent
	The Poisson process of order $i$ is a weighted sum of independent Poisson processes and is  used to model the flow of clients in different services. In the paper below we study some extensions of this process, for different forms of the weights and also with the time-changed versions, with Bern\v stein subordinator playing the role of time. We focus on the analysis of hitting times of these processes obtaining sometimes explicit distributions. 
	Since all the processes examined display a similar structure with multiple upward jumps sometimes they can skip all states with positive probability even on infinitely long time span.

%	arrival of packets of size $i$, where the $i$ items in packets are uniformly distributed. In this paper, we discuss the first hitting time probabilities of Poisson process of order $i$. We introduce some of its extensions and study their properties and hitting time probabilities. One such extension is to take a weighted sum of Poisson processes and another is by the means of time-changing the Poisson process of order $i$ by a L\'evy subordinator. Finally, we present a special case of our result and show that it matches with the results obtained by Orsingher and Polito (2012).  
\end{abstract}
\maketitle
\section{Introduction}

One of the most important processes in probability, capable of modelling different flows of events is the homogeneous or non-homogeneous Poisson process. Its widespread popularity stems from the fact that it has stationary and independent increments with exponentially distributed inter-arrival times. It is applicable for modelling of earthquakes, floods, arrival of clients in banks, car accidents and many other types of phenomenon.\\

% This model has elegance, simplicity and utility in application and therefore it remains a favourite choice of scientists in modelling different phenomena like the sequence of earthquakes, occurrence of floods and \textit{etc}.\\

\noindent Among its generalizations we have the Cox process (where the rate itself a stochastic process), the Poisson random fields and its time-changed versions which is  useful to model the wearing of machines (different when they are working or idle). Fractional extensions of the Poisson process have been introduced in the last two decades (see \cite{lask,sfpp,Maheshwari2019}) which include time and space fractional Poisson processes. Also state-dependent fractional Poisson processes have been studied (\cite{garra2015}). Long-range dependence property of the time-changed Poisson process is studied in \cite{LRD2014,biardapp,lrd2016} \\

\noindent In this paper, we are interested in the Poisson process of order $i$ and some of its extensions. In particular, we study the following three processes

% It was first studied as Poisson distribution of order $i$ by Philippou (1984) (see \cite{phil1984}) and later studied by Kostadinova and Minkova (2012) (see \cite{Poiss-order-k}) as the full-fledged Poisson process of order $i$. Recently, Sengar \textit{et. al} (see \cite{TCPPoK}) studied the time-changed Poisson process of order $i$. We here study  weighted sums of Poisson processes and  investigate their hitting time probabilities. We also study the hitting time probability of the weighted sums of the Poisson process subordinated by an independent L\'evy subordinator. The study of the Poisson process subordinated by L\'evy subordinators was done by Orsingher and Toaldo (2014) (see \cite{OrsToa-Berns}).\\\\
%In this work we study various types of combinations of Poisson processes, 

\begin{align}
	Y(t)&:=\sum_{j=1}^{i}jN_j(t)=\sum_{j=1}^{i}\sum_{k=j}^{i}N_k(t)\label{def:ppok}\\
	Z(t)&:=\sum_{j=1}^{i}g(j)N_j(t)\label{def:ppokg}
	\shortintertext{and }
	W(t)&:=\sum_{j=1}^{i}jN_j(H^f(t))\label{def:tc-ppok},
\end{align}
where $N_j(t),t>0,1\leq j\leq i$ are independent, homogeneous Poisson processes of rate $\lambda$, $g$ is an integer-valued, bounded, monotonically increasing function, $H^f(t)$ are subordinators related to a Bern\v stein function $f$, independent from $N_j$. \\\\%In the notation appearing in \eqref{def:tc-ppok} we explicitly denote the rate $\lambda$ of the Poisson process. \\\\
Recently, Sengar \textit{et. al} (see \cite{TCPPoK}) have studied the process $W(t)$, that is a time-changed Poisson process of order $i$. Furthermore, in this paper the authors have studied the process $Y(t)$, time-changed with the inverse of $H^f$. The process above take jumps of amplitude bigger than $i$ and in the case where $g$ is an integer-valued function the maximal jump has amplitude  equal to $g(i)$.\\\\
 Processes of the form \eqref{def:ppok} can arise in the study of queues in banks or post offices where different services are offered to clients.
%Assume that in a bank or in a post office $i$ different services are offered to clients. 
Since the types of services require a different amount of time to be carried out, the queues formed in front of each window is different. Therefore $jN_j(t)$ represents the random number of clients in the $j$-th queue while $Y(t)$ is the total random number of clients in the post office or bank agency. The process $Z(t)$ can be interpreted in the same way.\\

\noindent In this paper we study the hitting times for all processes, that is we study the distribution of the random variables
\begin{equation} \label{tk:def}
	T_k=\inf\{s:X(s)=k\},~k\geq 1.
\end{equation}
A similar analysis was made in the paper by Orsingher and Polito (2012) (see \cite{Orsingher2012}) for the iterated Poisson process, for the time-changed Poisson process by Orsingher and Toaldo (2015) (see \cite{OrsToa-Berns}) and also by Garra \textit{et. al} (2016) (see \cite{garra2015}) for the space fractional Poisson process. \\

\noindent In some cases it is possible to evaluate the hitting probabilities 
$$\mathbb{P}[T_k<\infty]$$
and, in particular for the process $Y(t),t>0$ we show that 
$$ \mathbb{P}[T_k<\infty]=
\left\{
\begin{array}{ll}
	\dfrac{k}{i}  & \mbox{if } k\leq i \\&\\
	1 & \mbox{if } k>i.
\end{array}
\right.$$
The same analysis is performed for the processes $Z(t)$ and $W(t)$. For the case where $H^f(t)$ is itself a Poisson process we extend the results obtained in \cite{Orsingher2012}.\\

\noindent One of our main concerns in the analysis of the hitting time $T_k$ defined in \eqref{tk:def}, where $X$ is one of the processes above. The random variables $T_k$ can be interpreted as the time necessary for the total queue to obtain for the first time the length $k$. We will show for $W$ that $\mathbb{P}[T_k<\infty]<1$ for all $k$ because all processes examined make upward jumps of arbitrary size with positive probability. \\

\noindent We will sometimes use the following notations.
%\noindent\textbf{Notations}

%	Let ${\bf x}=(x_1,x_2,\ldots,x_i)$ be a vector of  non-negative integers,
%		\begin{equation}\label{def:index}
%	 \begin{split}
%	\zeta_i :=& x_1+x_2+\ldots+x_i,\\
%	\Pi_i! := &x_1!x_2!\ldots x_i \text{ and }\\
%	\Omega(i,n):=&\left\{{\bf x} = (x_1,x_2,\ldots,x_i)\big|x_1+2x_2+\ldots+ix_i=n\right\}.
%	\end{split}
%	\end{equation}
	Let $N^{(i)}$ follow the Poisson distribution of order $i$ with rate parameter $\lambda>0$, then the probability mass function (\textit{pmf}) is given by
	\begin{equation*}\label{def:poiss-dist-k-1}
	\mathbb{P}[N^{(i)}=n]=\sum_{\substack{x_1,x_2,\ldots ,x_i\geq 0\\\sum_{j=1}^{i}jx_j=n}} e^{-i\lambda } \frac{\lambda^{x_1+\ldots+x_i}}{x_1!\ldots x_i!},~ n=0,1,\ldots,
	\end{equation*}
	where the summation is taken over all positive integers $x_1,x_2, \ldots,x_i$  such that $ x_1+2x_2+\ldots +ix_i = n$.

\section{Poisson process of order $i$}
The process defined in \eqref{def:ppok} during an infinitesimal interval of time of length $dt$ can perform jumps of size up to $\tfrac{i(i+1)}{2}$ and thus substantially differs from the classical Poisson process. This can be shown by observing that the probability generating function of $Y(t)$ reads
\begin{align}\label{pgf-yt}
	\mathbb{E}[u^{Y(t)}]=\mathbb{E}[u^{\sum_{j=1}^{i}jN_j(t)}]=\mathbb{E}[\Pi_{j=1}^{i}u^{jN_j(t)}]=\prod_{j=1}^{i}e^{-\lambda t(1-u^j)}=e^{-i\lambda t+\lambda t\sum_{j=1}^{i}u^j}.
\end{align}
For the increment $dY(t)$ we have therefore
\begin{align}
\mathbb{E}[u^{dY(t)}]&=e^{-i\lambda dt+\lambda dt\sum_{j=1}^{i}u^j}\\
&=(1-i\lambda dt)\left(1+\lambda dt\sum_{j=1}^{i}u^j+\ldots \right)\nonumber.
\end{align}
Therefore $Y(t)$ has the following distribution

\begin{align}\label{increments1}
		\left\{
	\begin{array}{ll}
	\mathbb{P}[dY(t)=j]=\lambda dt  & \mbox{for } 1\leq j\leq i \\
	\mathbb{P}[dY(t)=0]=1-i\lambda dt. & 
	\end{array}
	\right.
\end{align}
Since $Y(t)$ makes the jumps of length $j$ with uniform law in all time intervals of length $dt$ it can skip the first $i-1$ levels even during the infinite time horizon space.
 A similar behaviour was noted in many other cases like the iterated  Poisson process $N_1(N_2(t))$ where $N_k,k=1,2$ are independent homogeneous Poisson process (see \cite{Orsingher2012}), for the space-fractional Poisson process and its generalizations. In the present case, however, the first $k~(k<i)$ states are reached with positive probability.\\
In order to make more explicit the particular nature of $Y(t)$, we will show that the hitting time 
\begin{equation*}
	 T_k = \inf\left\{t: Y(t) = k \right\}= \inf\left\{t: \sum_{j=1}^{i}jN_j(t)=k \right\}
\end{equation*}
has the following form 
\begin{equation}\label{hitting1}
	\mathbb{P}[T_k < \infty]= \begin{cases} \dfrac{k}{i}~, 1\leq k \leq i-1 \\ 1 , ~ k\geq i, \end{cases}
\end{equation}
that it is initially increasing and then reaches a stationary form where all states are hit sooner or later with probability one.

\begin{remark}
	 For $i=1$, formula \eqref{hitting1} confirms that for the homogeneous Poisson processes all states are visited with probability one during an infinite time interval.
\end{remark}

%\noindent We pass now to the proof of \eqref{hitting1}. Because of the distribution \eqref{increments1} we can write
%\begin{align}
%	\mathbb{P}[T_k\in ds]&=\sum_{h=1}^{k}\mathbb{P}\left[\sum_{j=1}^{i}jN_j(s)=k-h,\sum_{j=1}^{i}jdN_j(s)=h \right]\nonumber\\
%	&=\sum_{h=1}^{k}\mathbb{P}\left[Y(s)=k-h\right]\mathbb{P}\left[dY(s)=h \right]\nonumber\\
%	&=\lambda ds\sum_{h=1}^{k}\mathbb{P}\left[Y(s)=k-h\right]=\lambda ds\sum_{n=0}^{k-1}\mathbb{P}\left[Y(s)=n\right]. \label{t_k}
%\end{align}
From \eqref{pgf-yt} we can infer that

\begin{equation}\label{pmf-yt}
	\mathbb{P}[Y(s)=n]=\sum_{\substack{x_1,x_2,\ldots ,x_i\geq 0\\\sum_{j=1}^{i}jx_j=n}} e^{-i\lambda s} \frac{(\lambda s)^{x_1+\ldots+x_i}}{x_1!\ldots x_i!}.
\end{equation}
Result \eqref{pmf-yt} can also be written down directly by applying convolution arguments.

\noindent The hitting time of a stochastic process $Y(t),t\geq 0$ be the first time at which the given process hits a level $k$, denoted as $T_k$, and is defined as
$$ T_k = \inf\{t: Y(t) = k \}. $$

\begin{theorem}
	Let $ N^{(i)}(t)$ be the Poisson process of order $i$ and $T_k$ be the first hitting time of $N^{(i)}(t)$, then
	$$\mathbb{P}[T_k < \infty]= \begin{cases} \frac{k}{i}~, 1\leq k \leq i-1 \\ 1 , ~ k\geq i \end{cases}.$$
\end{theorem}
\begin{proof}
	Let $\mathbb{P}[T_k \in ds],~ s>0 $ be the distribution function of $T_k$, then
 %and 
%$$ N^{(i)}(t) = N_1(t)+2N_2(t)+\ldots +kN_k(t),~in~distribution.$$
%Define,  $ T_i = \inf\{t: Y(t) = i \}$\\
%i.e, $ T_i = \inf\{t: N_1(t)+2N_2(t)+\ldots +iN_i(t) = i \} $\\
%Now, 
\begin{align}
\mathbb{P}[T_k \in ds] =& \sum_{h=1}^{k} \mathbb{P}[N^{(i)}(s)=k-h,N^{(i)}(s+ds)= k]\nonumber\\
=& \sum_{h=1}^{k} \mathbb{P}[N^{(i)}(s)=k-h,N^{(i)}(s)+N^{(i)}(ds)=k]\nonumber\\
=& \sum_{h=1}^{k} \mathbb{P}[N^{(i)}(s)=k-h] \mathbb{P}[N^{(i)}(ds)=h]\nonumber \\
%=& \sum_{h=1}^{i}\mathbb{P}(Y(s)=i-h)\mathbb{P}(dY(s)=h)\\
=& \sum_{h=1}^k \mathbb{P}[N^{(i)}(s)=k-h] \lambda ds,\label{tk-theorem}
\end{align}
where the last equality follows from 
\begin{align*}
\mathbb{E}[u^{N^{(i)}(ds)}] =& \exp[-i\lambda ds (1-\mathbb{E}u^X)]\\ =& 1-i\lambda ds (1-\mathbb{E}u^X) + o(ds)\\=& 1-i\lambda ds +i\lambda ds \left(\frac{1}{i}\sum_{j=1}^i u^j\right)\\=& 1-i\lambda ds + \lambda ds \left(\sum_{j=1}^i u^j\right).
\end{align*}
Hence, $$ \mathbb{P}[N^{(i)}(ds)=l]= \begin{cases} \lambda ds,~ 1\leq l \leq i\\ 1-i \lambda ds,~ l=0. \end{cases}$$
Using the above result in equation \eqref{tk-theorem}, we get
\begin{align*}
\mathbb{P}[T_k \in ds] =& \lambda \sum_{h=1}^k \mathbb{P}[N^{(i)}(s)=k-h] ds \\ =& \lambda[\mathbb{P}(N^{(i)}(s)=k-1) ds+\mathbb{P}(N^{(i)}(s)=k-2) ds+\ldots +\mathbb{P}(N^{(i)}(s)=0) ds]\\ =& \lambda \sum_{n=0}^{k-1}\mathbb{P}(N^{(i)}(s)=n) ds. 
\end{align*}
We now compute the hitting time probability of level $k$ of $N^{(i)}(t)$.\\
When $ k< i$, then
\begin{align*}
\mathbb{P}[T_k<\infty]=\int_0^{\infty} \mathbb{P}[T_k \in ds]=& \lambda\int_0^{\infty}\sum_{n=0}^{k-1}\mathbb{P}(N^{(k)}(s)=n) ds\\
=& \lambda\int_0^{\infty}\sum_{n=0}^{k-1} \left[ \sum_{\substack{x_1,x_2,\ldots ,x_i\geq 0\\\sum_{j=1}^{i}jx_j=n}} e^{-i\lambda s} \frac{(\lambda s)^{x_1+\ldots+x_i}}{x_1!\ldots x_i!}\right]ds\\
=& \lambda \sum_{n=0}^{k-1} \sum_{\substack{x_1,x_2,\ldots ,x_i\geq 0\\\sum_{j=1}^{i}jx_j=n}} \int_0^{\infty} e^{-i\lambda s} \frac{(\lambda s)^{x_1+\ldots+x_i}}{x_1!\ldots x_i!}ds \\
 =& \lambda \sum_{n=0}^{k-1} \sum_{\substack{x_1,x_2,\ldots ,x_i\geq 0\\\sum_{j=1}^{i}jx_j=n}}\frac{\lambda^{x_1+\ldots+x_i}}{x_1!\ldots x_i!}\int_0^{\infty} e^{-i\lambda s}s^{x_1+\ldots+x_i}ds \\
 =& \lambda \sum_{n=0}^{k-1} \sum_{\substack{x_1,x_2,\ldots ,x_i\geq 0\\\sum_{j=1}^{i}jx_j=n}}\frac{\lambda^{x_1+\ldots+x_i}}{x_1!\ldots x_i!} \frac{\Gamma(x_1+\ldots+x_i+1)}{(i\lambda)^{x_1+\ldots+x_i+1}}\\
 =& \sum_{n=0}^{k-1} \sum_{\substack{x_1,x_2,\ldots ,x_i\geq 0\\\sum_{j=1}^{i}jx_j=n}}\frac{\Gamma(x_1+\ldots+x_i+1)}{x_1!\ldots x_i!}i^{-(x_1+\ldots+x_i+1)}.
\end{align*}
Set $ x_j=n_j,~~1\leq j \leq i $ and $ n=x+\sum_{j=1}^i (j-1)n_j$ i.e. $ n_1+n_2+\ldots +n_i=x$ in the internal summation.

%%%%%%%%%%%%%%%%%%% start

\begin{align*}
\int_0^{\infty} \mathbb{P}[T_k \in ds] =& \sum_{n=0}^{k-1} \sum_{\substack{n_1,n_2,\ldots ,n_i\geq 0\\\sum_{j=1}^{i}n_j=x}}\frac{\Gamma(n_1+n_2+\ldots +n_i+1)}{n_1!n_2!\ldots n_i!}i^{-(n_1+n_2+\ldots +n_i+1)}\\
=& \sum_{n=0}^{k-1} \frac{1}{i} \sum_{\substack{n_1,n_2,\ldots ,n_i\geq 0\\\sum_{j=1}^{i}n_j=x}}\frac{x!}{n_1!n_2!\ldots n_i!}\frac{1}{k^{n_1+n_2+\ldots +n_i}}\\
 =& \sum_{n=0}^{k-1} \frac{1}{i} \left[\frac{1}{i}+\frac{1}{i}+\ldots +\frac{1}{i}\right]^x \\
 =& \sum_{n=0}^{k-1} \frac{1}{i} \\=& \frac{k}{i}< 1 \text{ for }k<i.
\end{align*}
When $ k\geq i$, then
\begin{align*}
\mathbb{P}[T_k \in ds] =& \sum_{h=k-i}^{k-1} \mathbb{P}(N^{(i)}(s)=k-h,N^{(i)}(s+ds)=k)\\
=& \sum_{h=k-i}^{k-1} \mathbb{P}(N^{(i)}(s)=k-h,N^{(i)}(s)+N^{(i)}(ds)=k)\\
=& \sum_{h=k-i}^{k-1} \mathbb{P}(N^{(i)}(s)=k-h) \mathbb{P}(N^{(i)}(ds)=h) \\ %=& \sum_{h=i-k}^{i-1}\mathbb{P}(N^{(i)}(s)=k-h)\mathbb{P}(dN^{(i)}(s)=h)\\
=& \sum_{h=k-i}^{k-1} \mathbb{P}(N^{(i)}(s)=k-h) \lambda ds\\
=& \lambda\sum_{m=1}^i \mathbb{P}[N^{(i)}(s)=m]ds.
\end{align*}
Now we have that
\begin{align*}
\int_0^{\infty} \mathbb{P}[T_k \in ds] =& \lambda\int_0^{\infty}\sum_{m=1}^i\mathbb{P}(N^{(i)}(s)=m) ds\\
=& \lambda\int_0^{\infty}\sum_{m=1}^i \left[ \sum_{\substack{x_1,x_2,\ldots ,x_i\geq 0\\\sum_{j=1}^{i}jx_j=n}} e^{-i\lambda s} \frac{(\lambda s)^{x_1+\ldots+x_i}}{x_1!\ldots x_i!}\right]ds\\
=& \lambda \sum_{m=1}^i \sum_{\substack{x_1,x_2,\ldots ,x_i\geq 0\\\sum_{j=1}^{i}jx_j=n}} \int_0^{\infty} e^{-i\lambda s} \frac{(\lambda s)^{x_1+\ldots+x_i}}{x_1!\ldots x_i!}ds \\ 
=& \lambda \sum_{m=1}^i \sum_{\substack{x_1,x_2,\ldots ,x_i\geq 0\\\sum_{j=1}^{i}jx_j=n}}\frac{\lambda^{x_1+\ldots+x_i}}{x_1!\ldots x_i!}\int_0^{\infty} e^{-i\lambda s}s^{x_1+\ldots+x_i}ds \\
=& \lambda \sum_{m=1}^i \sum_{\substack{x_1,x_2,\ldots ,x_i\geq 0\\\sum_{j=1}^{i}jx_j=n}}\frac{\lambda^{x_1+\ldots+x_i}}{x_1!\ldots x_i!} \frac{\Gamma(x_1+\ldots+x_i+1)}{(k\lambda)^{x_1+\ldots+x_i+1}}\\
=& \sum_{m=1}^i \sum_{\substack{x_1,x_2,\ldots ,x_i\geq 0\\\sum_{j=1}^{i}jx_j=n}}\frac{\Gamma(x_1+\ldots+x_i+1)}{x_1!\ldots x_i!}(k)^{-(x_1+\ldots+x_i+1)}.
\end{align*}
Set $ x_j=n_j,~~1\leq j \leq i $ and $ m=x+\sum_{j=1}^i (j-1)n_j$ i.e. $ n_1+n_2+\ldots +n_i=x$.
\begin{align*}
\int_0^{\infty} \mathbb{P}[T_k \in ds] =& \sum_{m=1}^i \sum_{\substack{n_1,n_2,\ldots ,n_i\geq 0\\		\sum_{j=1}^{i}n_j=x}}\frac{\Gamma(n_1+n_2+\ldots +n_i+1)}{n_1!n_2!\ldots n_i!}(k)^{-(n_1+n_2+\ldots +n_i+1)}\\
=& \sum_{m=1}^i \frac{1}{i} \sum_{\substack{n_1,n_2,\ldots ,n_i\geq 0\\\sum_{j=1}^{i}n_j=x}}\frac{x!}{n_1!n_2!\ldots n_i!}\frac{1}{i^{n_1+n_2+\ldots +n_i}}\\
 =& \sum_{m=1}^i \frac{1}{i} \left[\frac{1}{i}+\frac{1}{i}+\ldots +\frac{1}{i}\right]^x \\
 =&\sum_{m=1}^i\frac{1}{i} = \dfrac{i}{i} = 1.
\end{align*}
Hence, $$\mathbb{P}[T_k< \infty]= \begin{cases} \frac{k}{i}~, 1\leq k \leq i-1 \\ 1 , ~ k\geq i \end{cases}.$$
\end{proof}
\begin{remark}
	The above theorem states that the Poisson process of order $i$ hits every level $k$ when $k$ is bigger than $i$ and with probability $k/i$ it will hit level $k$ if $k$ is less than $i$. This  is consistent with the fact that arrivals are in packets of size $i$.
\end{remark}
Since the maximal size of jumps of process $Y(t)$ is $i$, all states larger or equal $i$ can be obtained with probability one, while the process can avoid with positive probability the first $i-1$ states.
We have $ N_1(t)+2N_2(t)+\ldots  +iN_i(t) = \sum_{j=1}^{N(t)}X_j^i, $  where $X_j^i$ are independent integer-valued random variables uniformly distributed on the set $\{1,2,\ldots,i\}$.

\section{Weighted sum of Poisson processes}
In this section we study the behaviour of the process
\begin{equation}\label{process-2}
	Z(t)=\sum_{j=1}^{i}g(j)N_j(t),
\end{equation}
where $N_j, j=1,\ldots,n$ are independent, homogeneous Poisson processes and $g$ is an integer-valued, bounded and increasing function. We now give the characterisation of the weighted sum of Poisson processes in the following theorem.
\begin{theorem}
	The process $Z$ defined in \eqref{process-2} can be represented as a compound Poisson processes, that is
	\begin{equation}
Z(t)=\sum_{j=1}^{i}g(j)N_j(t)\stackrel{d}{=}\sum_{j=1}^{N(t)}g(X_j^i),	
	\end{equation}
 where the random variables $X_j^i$ are independent, discrete and uniform on the set $\{1,2,\ldots,i\}$ and $N(t)$ is Poisson with parameter $i\lambda$.
\end{theorem}
\begin{proof}
The probability generating function (\textit{pgf}) of $Z(t)$ reads%\\
%Let
% $$ Z(t)=\sum_{j=1}^{i}g(j)N_j(t), $$
%where $g$ is an integer-valued monotonically increasing function and $N(t)$ is a Poisson process with rate $i\lambda$. Then
\begin{align*}
\mathbb{E}[u^{Z(t)}] =& \mathbb{E}[u^{\sum_{j=1}^{i}g(j)N_j(t)}]\\ =& \prod_{j=1}^i \mathbb{E}[u^{g(j)N_j(t)}] \\ 
=& \prod_{j=1}^i \exp[-\lambda t+ \lambda t u^{g(j)}]\\
 =& \exp[-i\lambda t+ \lambda t \sum_{j=1}^i u^{g(j)}].
\end{align*}
Now we will find the \textit{pgf} of R.H.S.
\begin{align*}
\mathbb{E}u^{\sum_{j=1}^{N(t)}g(X_j^i)} =& \sum_{n=0}^{\infty}\mathbb{E}[u^{\sum_{j=1}^{N(t)}g(X_j^i)} |N(t)=n] \mathbb{P}(N(t)=n)\\ 
=& \sum_{n=0}^{\infty}\mathbb{E}[u^{\sum_{j=1}^{n}g(X_j^i)}] \mathbb{P}(N(t)=n)\\ 
=& \sum_{n=0}^{\infty} \prod_{j=1}^n \mathbb{E}u^{g(X_j)}\mathbb{P}(N(t)=n)\\
 =&  \sum_{n=0}^{\infty} (\mathbb{E}u^{g(X)})^n e^{-i\lambda t}\frac{(i\lambda t)^n}{n!}\\ =& \exp(-i\lambda t)\exp\left[i\lambda t \mathbb{E}[u^{g(X)}]\right]\\
  =& \exp\left[-i\lambda t(1-\mathbb{E}[u^{g(X)}])\right]\\
 =& \exp\left[-i\lambda t\left(1-\frac{1}{i}\sum_{j=1}^iu^{g(j)}\right)\right]\\ =& \exp\left[-i\lambda t+ \lambda t \sum_{j=1}^i u^{g(j)}\right],
\end{align*}
and this completes the proof.
\end{proof}
\begin{remark}
For $g(j)=j$, we retrieve the \textit{pgf} of $Y(t)$. Furthermore
%The distribution of the random variable $\sum_{j=1}^{N(t)}g(X_j^i)$ writes as
$$ \mathbb{P}[Z(t)=n]= \sum_{\substack{x_1,\ldots ,x_i\geq 0\\\sum_{j=1}^{i}g(j)x_j=n}} e^{-i\lambda t} \frac{ (\lambda t)^{x_1+\ldots+x_i}}{x_1!\ldots x_i!}.$$
\end{remark}
We pass now to the analysis of the hitting times of $T_k,k\geq 1$ of the process $Z(t)$. In the following theorem, we compute the distribution of $T_k$.
%
%Now, Let $ Z(t) = g(1)N_1(t)+g(2)N_2(t)+\ldots +g(i)N_i(t)$\\
%when $ k< g(i)$,

\begin{theorem}
The hitting time $T_k,k\geq 1$ has distribution 

\begin{equation}
\mathbb{P}[T_k \in ds] =\lambda ds\sum_{\substack{{h=1}\\ g(h)\leq k}}^k\sum_{\substack{x_1,\ldots ,x_i\geq 0\\\sum_{j=1}^{i}g(j)x_j=k-g(h)}} e^{-i\lambda t} \frac{ (\lambda t)^{x_1+\ldots+x_i}}{x_1!\ldots x_i!},~s>0.
\end{equation}

\end{theorem}
\begin{proof}
Recall that 
\begin{equation*}
	Z(t)=g(1)N_1(t)+g(2)N_2(t)+\ldots +g(i)N_i(t).
\end{equation*}We begin with
\begin{align*}
\mathbb{P}[T_k \in ds] =& \sum_{\substack{{h=1}\\ g(h)\leq i}}^i \mathbb{P}[Z(s)=k-g(h),Z(s+ds)= k]\\
=& \sum_{\substack{{h=1}\\ g(h)\leq k}}^k \mathbb{P}[Z(s)=k-g(h),Z(s)+dZ(s)=k]\\
=& \sum_{\substack{{h=1}\\ g(h)\leq k}}^k \mathbb{P}[Z(s)=k-g(h)]\mathbb{P}[dZ(s))=g(h)] \\
%=& \sum_{\substack{{h=1}\\ g(h)\leq k}}^k\mathbb{P}(N^{(i)}(s)=k-g(h))\mathbb{P}(dN^{(i)}(s)=g(h))\\
=& \sum_{\substack{{h=1}\\ g(h)\leq k}}^k \mathbb{P}[Z(s)=k-g(h)] \lambda ds\\
=&\lambda ds\sum_{\substack{{h=1}\\ g(h)\leq k}}^k\sum_{\substack{x_1,\ldots ,x_i\geq 0\\\sum_{j=1}^{i}g(j)x_j=k-g(h)}} e^{-i\lambda s} \frac{ (\lambda s)^{x_1+\ldots+x_i}}{x_1!\ldots x_i!},
\end{align*}
which completes the proof.
\end{proof}
We now find the probability that $Z$ hits an arbitrary  level $k$.
\begin{theorem}
The probability of hitting any level $k$ is given by
\begin{equation}
	\mathbb{P}[T_k < \infty]= \begin{cases} \dfrac{\#(h:1\leq g(h)\leq k)}{i}~, 1\leq k <g(i)\\ 1 , ~ k\geq g(i) \end{cases}.
\end{equation}
\end{theorem}
\begin{proof}
Let $1\leq k< g(i)$. We have that
\begin{align*}
\mathbb{P}[T_k < \infty]=&\int_0^{\infty} \mathbb{P}[T_k \in ds] = \int_0^{\infty}\sum_{\substack{{h=1}\\ g(h)\leq k}}^k \mathbb{P}(Y(s)=k-g(h)) \lambda ds\\=& \lambda\int_0^{\infty}\sum_{\substack{{h=1}\\ g(h)\leq k}}^k \left[ \sum_{\substack{x_1,x_2,\ldots ,x_i\geq 0\\\sum_{j=1}^{k}g(j)x_j=k-g(h)}} e^{-i\lambda s} \frac{(\lambda s)^{x_1+\ldots+x_i}}{x_1!\ldots x_i!}\right]ds\\
=& \lambda \sum_{\substack{{h=1}\\ g(h)\leq k}}^k \sum_{\substack{x_1,x_2,\ldots ,x_i\geq 0\\\sum_{j=1}^{i}g(j)x_j=k-g(h)}} \int_0^{\infty} e^{-i\lambda s} \frac{(\lambda s)^{x_1+\ldots+x_i}}{x_1!\ldots x_i!}ds \\ 
=& \lambda \sum_{\substack{{h=1}\\ g(h)\leq k}}^k \sum_{\substack{x_1,x_2,\ldots ,x_i\geq 0\\\sum_{j=1}^{i}g(j)x_j=k-g(h)}}\frac{\lambda^{x_1+\ldots+x_i}}{x_1!\ldots x_i!}\int_0^{\infty} e^{-i\lambda s}s^{x_1+\ldots+x_i}ds \\
=& \lambda \sum_{\substack{{h=1}\\ g(h)\leq k}}^k \sum_{\substack{x_1,x_2,\ldots ,x_i\geq 0\\\sum_{j=1}^{i}g(j)x_j=k-(h)}}\frac{\lambda^{x_1+\ldots+x_i}}{x_1!\ldots x_i!} \frac{\Gamma(x_1+\ldots+x_i+1)}{(i\lambda)^{x_1+\ldots+x_i+1}}\\
=& \sum_{\substack{{h=1}\\ g(h)\leq k}}^k \sum_{\substack{x_1,x_2,\ldots ,x_i\geq 0\\\sum_{j=1}^{i}g(j)x_j=k-g(h)}}\frac{\Gamma(x_1+\ldots+x_i+1)}{x_1!\ldots x_i!}k^{-(x_1+\ldots+x_i+1)}\\ 
=& \sum_{\substack{{h=1}\\ g(h)\leq k}}^k \frac{1}{i} \\
 =& \frac{\#(h:1\leq g(h)\leq k)}{i}.
\end{align*}
Now assume $ k \geq g(i)$, we get distribution of $T_k$ as
\begin{align}
\mathbb{P}[T_k \in ds] =&\sum_{h=k-g(i)}^{k-g(1)} \mathbb{P}[Z(s)=k-h,Z(s+ds)=k]\nonumber\\
%=&\sum_{h=k-g(i)}^{k-g(1)}\mathbb{P}[Z(s)=k-h,Z(s+ds)=k]\\
=& \sum_{h=k-g(i)}^{k-g(1)}\mathbb{P}[Z(s)=k-h,Z(s)+dZ(ds)=k]\nonumber\\
=& \sum_{h=k-g(i)}^{k-g(1)}\mathbb{P}[Z(s)=k-h]\mathbb{P}[dZ(s)=h]\nonumber\\
=&\lambda ds\sum_{h=k-g(i)}^{k-g(1)}\mathbb{P}[Z(s)=k-h]\nonumber\\
=& \lambda ds \sum_{h=g(1)}^{g(i)}\mathbb{P}[Z(s)=h]ds. \label{ws-tk1}
\end{align}
We now obtain the hitting time probability as follows
\begin{align}
\mathbb{P}[T_k<\infty] =& \int_0^{\infty}\mathbb{P}[T_k \in ds]\nonumber\\
=&\lambda\int_0^{\infty}\sum_{h=g(1)}^{g(i)}\mathbb{P}[Z(s)=h]ds\nonumber\\
=& \lambda\int_0^{\infty}\sum_{h=g(1)}^{g(i)} \left[ \sum_{\substack{x_1,\ldots ,x_i\geq 0\\\sum_{j=1}^{i}g(j)x_j=h}} e^{-i\lambda s} \frac{ (\lambda s)^{x_1+\ldots+x_i}}{x_1!\ldots x_i!} \right]\nonumber\\
=& \lambda \sum_{h=g(1)}^{g(i)} \sum_{\substack{x_1,\ldots ,x_i\geq 0\nonumber\\
\sum_{j=1}^{i}g(j)x_j=h}}\frac{\lambda^{x_1+\ldots+x_i}}{x_1!\ldots x_i!}\int_0^{\infty}e^{-i\lambda s} s^{x_1+\ldots+x_i}ds\nonumber\\
=& \sum_{h=g(1)}^{g(i)}\sum_{\substack{x_1,\ldots ,x_i\geq 0.\nonumber\\
\sum_{j=1}^{i}g(j)x_j=h}}\frac{x_1+\ldots+x_i!}{{x_1!\ldots x_i!}}i^{-(x_1+\ldots+x_i+1)}
& \intertext{Set $x=x_1+\ldots+x_i,x_j=n_j,1\leq j\leq i$ and $h=x+\sum_{j=1}^{i}[g(j)-1]x_j$ in the inner summation}
=& \sum_{h=g(1)}^{g(i)} \sum_{\substack{n_1,\ldots ,n_i\geq 0\\
\sum_{j=1}^{i}n_j=x}}\frac{(n_1+\ldots +n_i)!}{{n_1!\ldots  n_i!}}i^{-(n_1+\ldots +n_i+1)}= \sum_{h=g(1)}^{g(i)} \frac{1}{i}= 1. \label{ws-tk2}
\end{align}
From \eqref{ws-tk1} and \eqref{ws-tk2}, we complete the proof.
\end{proof}
\section{Weighted sums of subordinated Poisson processes}
In this section, we study the weighted sums of subordinated Poisson process time-changed by the L\'evy subordinator. We derive its characterisation in the form of the compound Poisson process and compute its hitting time probabilities. The study of the of subordinated Poisson process is done by Maheshwari and Vellaisamy (2019) (see \cite{Maheshwari2019}) and results about their hitting time probabilities can be found in Orsingher and Toaldo (2015) (see \cite{OrsToa-Berns}). We now briefly introduce the L\'evy subordinator.

Let $f$ be a Bern\v stein function with integral representation 
\begin{equation}
f(x)=\int_{0}^{\infty}(1-e^{-sx})\nu(ds),
\end{equation}
where $\nu$ is the L\'evy measure, that is a non-negative measure such that
\begin{equation}
\int_{0}^{\infty}\min(s,1)\nu(ds)<\infty.
\end{equation}
An alternative definition of Bern\v stein function is that, for all $n$ it happens that
$$(-1)^nf^{(n)}(x)\leq 0.$$
We denote by $H^f$ the subordinator related to the Bern\v stein function $f$ or to the related L\'evy measure $\nu$. \\
The Laplace transform of $H^f$ reads
\begin{equation}
\mathbb{E}[e^{-\mu H^f(t)}]=e^{-tf(\mu)}.
\end{equation}
For example, if $f(\mu)=\mu^\alpha,0<\alpha<1,~\nu(dx)=\frac{\alpha \mu^{\alpha-1}}{\Gamma(1-\alpha)}$, we have the stable subordinator of order $\alpha$.
We now prove the equality in distribution of $W$ with its compound Poisson representation.

\begin{theorem}The following identity holds
	\begin{equation}\label{thm1:sec4}
		W(t)=\sum_{j=1}^{i}jN_j(H^f(t),\lambda)\stackrel{d}{=}\sum_{j=1}^{N(H^f(t),i\lambda)}X_j^i,
	\end{equation}
where $X_j,j=1,2\ldots$ are independent, uniformly distributed random variables (on the set of natural numbers $1\leq j\leq i$), $N_j(\bullet,\lambda)$ are independent homogeneous Poisson processes as well as $N_j(\bullet,i\lambda)$. Of course $H^f(t)$ are L\'evy subordinators related to the Bern\v stein function $f$.

\end{theorem}
\begin{proof}
	The probability generating function of \eqref{thm1:sec4} reads

\begin{align*}
\mathbb{E}[u^{\sum_{j=1}^{N(H^f(t),i\lambda)}X_j^i}] =& \mathbb{E}[\mathbb{E}[u^{\sum_{j=1}^{N(H^f(t),i\lambda)}X_j^i}|H^f(t)]]\\
=& \mathbb{E}[e^{-i\lambda H^f(t)+\lambda H^f(t)(u+u^2+\ldots +u^i)}]\\
=&\mathbb{E}[e^{-[i\lambda-\lambda(u+u^2+\ldots +u^i)]H^f(t)}]\\
=&e^{-tf[i\lambda-\lambda(u+u^2+\ldots +u^i)]}.
\end{align*}
On the other hand
\begin{align*}
\mathbb{E}[u^{\sum_{j=1}^i jN_j(H^f(t))}] =&
\int_{0}^{\infty} \mathbb{E}[u^{\sum_{j=1}^i
jN_j(H^f(t))}|H^f(t)=y]\mathbb{P}[H^f(t)\in dy]\\
=&\int_{0}^{\infty} \mathbb{E}[u^{\sum_{j=1}^i jN_j(y)}]
\mathbb{P}[H^f(t)\in dy]\\
=&\int_{0}^{\infty} e^{-i\lambda y+\lambda y(u+u^2+\ldots +u^i)} \mathbb{P}[H^f(t)\in dy]\\
=& \mathbb{E}[e^{-[i\lambda-\lambda(u+u^2+\ldots +u^i)]H^f(t)}]\\
=&e^{-tf[i\lambda-\lambda(u+u^2+\ldots +u^i)]}
\end{align*}
and this concludes the proof.
\end{proof}
\noindent We now pass to the analysis of the hitting times of the process  $W(t)=\sum_{j=1}^{i}jN_j(H^f(t))$. We observe that also in this case we can write for 
$$T_k=\inf\{s<t:W(s)=k\}$$
that
\begin{align*}
\mathbb{P}[T_k \in ds] =& \sum_{h=1}^k \mathbb{P}[W(s)=k-h, W(s+ds)=k]\\
=&\sum_{h=1}^k \mathbb{P}[W(s)=k-h,N_1(H^f(s)+dH^f(s))+\ldots +iN_i(H^f(s)+dH^f(s))=k]\\
=&\sum_{h=1}^k \mathbb{P}[W(s)=k-h, N_1(dH^f(s))+\ldots +iN_i(dH^f(s))=h]\\
=& \sum_{h=1}^k
\mathbb{P}[W(s)=k-h] \mathbb{P}[dW(s)=h],
\end{align*}
where 
\begin{equation}\label{ws:incre}
\mathbb{P}[dW(s)=h] = \begin{cases} 1-(ds)f(i \lambda)+ o(h),&h=0 \\ 
&\\(-ds)\left(\sum\limits_{\substack{x_1,x_2,\ldots ,x_i\geq 0\\\sum_{j=1}^{i}jx_j=h}} \frac{ (- \lambda )^{x_1+\ldots+x_i}}{x_1!\ldots x_i!}\frac{1}{i^{x_1+\ldots+x_i}}\frac{d^{(x_1+\ldots+x_i)}}{d\lambda^{(x_1+\ldots+x_i)}}f(i\lambda)\right)  +o(ds),&h=1,2,\ldots  \end{cases}. 
\end{equation}
and
\begin{align*}
\mathbb{P}[W(s)=k-h]=& \int_0^{\infty}\mathbb{P}[W(s)=k-h|H^f(s)=y]\mathbb{P}[H^f(s)\in dy]\\
=& \int_0^{\infty}\mathbb{P}[N_1(y)+\ldots +iN_i(y)=k-h]\mathbb{P}[H^f(s)\in dy]\\
=&\int_0^{\infty}\sum_{\substack{x_1,\ldots ,x_i\geq 0\\\sum_{j=1}^{i}jx_j=k-h}}e^{-i\lambda y} \frac{ (\lambda y)^{x_1+\ldots +x_i}}{x_1!\ldots  x_i!}\mathbb{P}[H^f(s)\in dy]\\
=&\sum_{\substack{x_1,\ldots ,x_i\geq 0\\\sum_{j=1}^{i}jx_j=k-h}} \frac{ \lambda^{x_1+\ldots +x_i}}{x_1!\ldots  x_i!}\int_0^{\infty} e^{-i\lambda y} y^{x_1+\ldots +x_i}\mathbb{P}[H^f(s)\in dy]\\
=&\sum_{\substack{x_1,\ldots ,x_i\geq 0\\\sum_{j=1}^{i}jx_j=k-h}} \frac{(- \lambda)^{x_1+\ldots +x_i}}{x_1!\ldots  x_i!}\frac{1}{i^{x_1+\ldots+x_i}}\int_0^{\infty}\dfrac{d^{x_1+\ldots +x_i}(e^{-i\lambda y})}{d\lambda^{x_1+\ldots +x_i}}\mathbb{P}[H^f(s)\in dy]\\
=& \sum_{\substack{x_1,\ldots ,x_i\geq 0\\\sum_{j=1}^{i}jx_j=k-h}} \frac{(- \lambda)^{x_1+\ldots +x_i}}{x_1!\ldots  x_i! i^{x_1+\ldots+x_i}}\dfrac{d^{x_1+\ldots +x_i}}{d\lambda^{x_1+\ldots +x_i}}\int_0^{\infty}e^{-i\lambda y} \mathbb{P}[H^f(s)\in dy]\\
=& \sum_{\substack{x_1,\ldots ,x_i\geq 0\\\sum_{j=1}^{i}jx_j=k-h}} \frac{(- \lambda)^{x_1+\ldots +x_i}}{x_1!\ldots  x_i!i^{x_1+\ldots+x_i}}\dfrac{d^{x_1+\ldots +x_i}}{d\lambda^{x_1+\ldots +x_i}}\mathbb{E}[e^{-i\lambda H^f(s)}]\\
=& \sum_{\substack{x_1,\ldots ,x_i\geq 0\\\sum_{j=1}^{i}jx_j=k-h}} \frac{(- \lambda)^{x_1+\ldots +x_i}}{x_1!\ldots  x_i!i^{x_1+\ldots+x_i}}\dfrac{d^{x_1+\ldots +x_i}}{d\lambda^{x_1+\ldots +x_i}}e^{-sf(i\lambda)}.
\end{align*}
In conclusion the hitting time of the state $k$ has probability 
\begin{align*}
\mathbb{P}[T_k<\infty]&=\sum_{h=1}^{k} \sum_{\substack{x_1,\ldots ,x_i\geq 0\\\sum_{j=1}^{i}jx_j=k-h}} \frac{(- \lambda)^{x_1+\ldots +x_i}}{x_1!\ldots  x_i!i^{x_1+\ldots+x_i}}\dfrac{d^{x_1+\ldots +x_i}}{d\lambda^{x_1+\ldots +x_i}}\frac{1}{f(i\lambda)}\\ &~~~~\times\sum\limits_{\substack{x_1,\ldots ,x_i\geq 0\\\sum_{j=1}^{i}jx_j=h}} \frac{ (- \lambda )^{x_1+\ldots+x_i}}{x_1!\ldots x_i!}\frac{d^{(x_1+\ldots+x_i)}}{d\lambda^{(x_1+\ldots+x_i)}}f(i\lambda),
\end{align*}
as shown by integrating with respect to $s$ and substituting \eqref{ws:incre} in the expression of $\mathbb{P}[W(s)=k-h]$. This result shows that the explicit distribution of the hitting time $T_k$ crucially depends on the Bern\v stein  function $f$ as pointed out in \cite{garra2015} and \cite{OrsToa-Berns}.\\
We work a special case of the results obtained in here in the next section.
\section{A Special case}
We here study the process 
\begin{equation}\label{ut-process}
U(t)=\sum_{j=1}^{i}jN_j^\lambda(N^\beta(t)),
\end{equation}
which is a special case of $W(t)$, where the time change is performed by means of a homogeneous Poisson process of rate $\beta$ independent from $N_j^\lambda,1\leq j\leq i$. We prove in the next section that $U(t),t>0$ can be represented in the form of a compound Poisson as shown in the next Theorem.
\begin{theorem} The following relationship holds
	\begin{equation}\label{thm-splcase}
U(t)=\sum_{j=1}^{i}jN_j^\lambda(N^\beta(t))\stackrel{d}{=} \sum_{j=1}^{N^{i\lambda}(N^{\beta}(t))}X_j^i,
	\end{equation}
where the $X_j^i$ are independent random variables, uniformly distributed on the set $\{1,\ldots,i\}$ and $N^{i\lambda}(t)$ is a Poisson process of rate $i\lambda$ independent from $N^\beta(t)$ and independent from all the $X^i_j$ random variables.
\end{theorem}
\begin{proof}
The probability generating function of the right-hand side of \eqref{thm-splcase} is
\begin{align*}
\mathbb{E}[u^{\sum_{j=1}^{N^{i\lambda}(N^{\beta}(t))}X_j^i}] =& \sum_{k=0}^{\infty}\left(\mathbb{E}u^X\right)^k\sum_{r=0}^{\infty}\mathbb{P}[N^\lambda(r)=k]\mathbb{P}[N^\beta(t)=r]\\
=&\sum_{r=0}^{\infty}e^{-\beta t}\frac{(\beta t)^r}{r!}e^{-i\lambda r+\lambda ir\mathbb{E}u^X}\\
=&\sum_{r=0}^{\infty}e^{-\beta t}\frac{(\beta t)^r}{r!}e^{-i\lambda r+\lambda  r(u^1+\ldots+u^i)}\\
=&e^{-\beta t+\beta t(e^{-[i\lambda-\lambda(u^1+\ldots +u^i)]})}.
\end{align*}
The probability generating function of the left-hand side of \eqref{thm-splcase} writes
\begin{align*}
\mathbb{E}[u^{\sum_{j=1}^i jN_j^\lambda(N^{\beta}(t))}] =&
\sum_{k=0}^{\infty}\prod_{j=1}^{i} \mathbb{E}[u^{jN_j^\lambda(N^{\beta}(t))}|N^{\beta}(t)=k]\mathbb{P}[N^{\beta}(t)=k]\\
=&\sum_{k=0}^{\infty}\prod_{j=1}^{i} e^{-\lambda k(1-u^j)}e^{-\beta t}\frac{(\beta t)^k}{k!}\\
=&\sum_{k=0}^{\infty} e^{-\lambda ki+\lambda k(u^1+\ldots+u^j)}e^{-\beta t}\frac{(\beta t)^k}{k!}\\
=&e^{-\beta t+\beta t(e^{-[i\lambda-\lambda(u^1+\ldots +u^i)]})},
\end{align*}
and this concludes the proof.
\end{proof}

\begin{remark}
In view of \eqref{pgf-yt}, we can write down the distribution  of process \eqref{ut-process} as 
\begin{align}
\mathbb{P}[U(t)=m]&=\sum_{r=1}^{\infty}\mathbb{P}\left[\sum_{j=1}^{i}jN_j^\lambda(r)=m\right]\mathbb{P}\left[N^\beta(t)=r\right],m\geq 1\nonumber\\
&=\sum_{r=1}^{\infty}e^{-\beta t}\frac{(\beta t)^r}{r!}\sum_{\substack{x_1,x_2,\ldots ,x_i\geq 0\\\sum_{j=1}^{i}jx_j=m}} e^{-i\lambda r} \frac{(\lambda r)^{x_1+\ldots+x_i}}{x_1!\ldots x_i!}\nonumber\\
&=\sum_{\substack{x_1,x_2,\ldots ,x_i\geq 0\\\sum_{j=1}^{i}jx_j=m}}\frac{\lambda ^{x_1+\ldots+x_i}}{x_1!\ldots x_i!}\mathbb{E}\left[e^{-i\lambda N^\beta(t)}(N^\beta (t))^{x_1+\ldots+x_i}\right].\label{pmf-ut}
\end{align}
\end{remark}
\noindent We now consider the hitting time distribution of the process $U(t),t>0$.\\
We can write that 
\begin{equation}\label{spl-tk}
\mathbb{P}[T_k\in ds]=\sum_{h=1}^{k}\mathbb{P}[U(s)=k-h]\mathbb{P}[dU(s)=k].
\end{equation}
Note that the increment of the iterated Poisson process $N^\alpha(N^\beta (t))$ has distribution

\begin{equation}
\mathbb{P}[dN^\alpha(N^\beta (t))=r] =
\left\{
\begin{array}{ll}
\beta  e^{-\alpha}ds+1-\beta ds  & \mbox{if } r= 0 \\&\\
e^{-\alpha}\dfrac{\alpha}{r!}\beta ds & \mbox{if } r\geq 1
\end{array}
\right.
\end{equation}
as can be ascertained by observing that 
$$\mathbb{P}[dN^\alpha(N^\beta (t))=r]=\sum_{m=1}^{\infty}\mathbb{P}[dN^\alpha(m)=r]\mathbb{P}[N^\beta(t)=m] $$
and that $dN^\beta (t)=0$ implies that $N^\alpha(0)=0$. The increment of the process $U(t)$ has distribution 
\begin{equation}\label{spl-utincr}
\mathbb{P}[dU(t)=r] =
\left\{
\begin{array}{ll}
1-\beta dt+e^{-\alpha i}\beta dt  & \mbox{if } r= 0 \\&\\
\beta dt\mathbb{P}\left[\sum_{j=1}^{i}jN_j^\alpha(1)=r\right] & \mbox{if } r\geq 1.
\end{array}
\right.
\end{equation}
From \eqref{pmf-yt} we have that 
\begin{equation}
\mathbb{P}\left[\sum_{j=1}^{i}jN_j^\alpha(1)=r\right]=\sum_{\substack{x_1,x_2,\ldots ,x_i\geq 0\\\sum_{j=1}^{i}jx_j=r}} e^{-i\alpha } \frac{\alpha^{x_1+\ldots+x_i}}{x_1!\ldots x_i!}.
\end{equation}
From \eqref{spl-tk} and formulas \eqref{pmf-ut} and \eqref{spl-utincr} we can write down the distribution of the hitting time, for $k\geq 1$
\begin{equation}
\mathbb{P}[T_k\in ds]=\sum_{h=1}^{k}\sum_{\substack{x_1,x_2,\ldots ,x_i\geq 0\\\sum_{j=1}^{i}jx_j=k-h}}\frac{\lambda ^{x_1+\ldots+x_i}}{x_1!\ldots x_i!}\mathbb{E}\left[e^{-i\lambda N^\beta(s)}(N^\beta (s))^{x_1+\ldots+x_i}\right]\beta ds\mathbb{P}\left[\sum_{j=1}^{i}jN_j^\alpha(1)=h\right].
\end{equation}
By integrating with respect to the time $s$ we get the hitting probability of the state $k$ as 
\begin{align}\label{hitting-ut}
\mathbb{P}[T_k<\infty] =&\sum_{h=1}^{k}\sum_{\substack{x_1,x_2,\ldots ,x_i\geq 0\\\sum_{j=1}^{i}jx_j=k-h}}\frac{\lambda ^{x_1+\ldots+x_i}}{x_1!\ldots x_i!}\sum_{r=0}^{\infty}e^{-i\lambda r}r^{x_1+\ldots+x_i}\mathbb{P}\left[\sum_{j=1}^{i}jN_j^\alpha(1)=h\right]
\end{align}
because 
\begin{align*}
\int_{0}^{\infty}\mathbb{E}\left[e^{-i\lambda N^\beta(s)}(N^\beta (s))^{x_1+\ldots+x_i}\right]ds&=\int_{0}^{\infty}\sum_{r=0}^{\infty}e^{-i\lambda r}r^{x_1+\ldots+x_i}\mathbb{P}\left[N^\beta(s)=r\right]ds\\
&=\frac{1}{\beta}\sum_{r=0}^{\infty}e^{-i\lambda r}r^{x_1+\ldots+x_i}.
\end{align*}
We are able to evaluate \eqref{pmf-ut} for $k=1$ and obtain
\begin{align*}
\mathbb{P}[T_1<\infty] =& \sum_{\substack{x_1,x_2,\ldots ,x_i\geq 0\\\sum_{j=1}^{i}jx_j=0}} \frac{\lambda^{x_1+\ldots+x_i}}{x_1!\ldots x_i!}\sum_{r=0}^{\infty}e^{-i\lambda r}\mathbb{P}\left[\sum_{j=1}^{i}jN_j^\alpha(1)=1\right]\\
=&\sum_{r=0}^{\infty}e^{-i\lambda r}\sum_{\substack{x_1,x_2,\ldots ,x_i\geq 0\\\sum_{j=1}^{i}jx_j=1}}e^{-i\lambda} \frac{\lambda^{x_1+\ldots+x_i}}{x_1!\ldots x_i!}\\
=&  e^{-i\lambda}\sum_{y=0}^{\infty}e^{-i\lambda y} \lambda= \lambda e^{-i\lambda} \frac{1}{1-e^{-i\lambda}}<1.
\end{align*}
This result coincides with equation (37) of Orsingher and Polito (2012) (see \cite{Orsingher2012}) for a rate $\lambda_\alpha=i\lambda$. For the hitting time $T_2$ we have that (see formula \eqref{hitting-ut})
\begin{align*}
\mathbb{P}[T_2<\infty] =& \sum_{r=0}^{\infty}e^{-i\lambda r}\mathbb{P}\left[\sum_{j=1}^{i}jN_j^\alpha(1)=2\right]+ \sum_{r=1}^{\infty}e^{-i\lambda r}\lambda r\mathbb{P}\left[\sum_{j=1}^{i}jN_j^\alpha(1)=1\right]\\
=& \sum_{r=0}^{\infty}e^{-i\lambda r}\left\{\lambda e^{-i\lambda}+\frac{\lambda^2}{2}e^{-i\lambda}\right\}+ \sum_{r=1}^{\infty}e^{-i\lambda r}\lambda r\lambda e^{-i\lambda}\\
=& e^{-i\lambda }\left(\lambda +\frac{\lambda^2}{2}\right)\frac{1}{1-e^{-i\lambda}}+ \lambda^2e^{-i\lambda}\sum_{r=1}^{\infty}re^{-i\lambda r}\\
=& \left(\lambda +\frac{\lambda^2}{2}\right)\frac{e^{-i\lambda}}{1-e^{-i\lambda}}+ \frac{\lambda^2e^{-2i\lambda}}{(1-e^{-i\lambda})^2}\\
=&\left(1+\frac{\lambda}{2}\right)\mathbb{P}[T_1<\infty]+\left(\mathbb{P}[T_1<\infty]\right)^2>\mathbb{P}[T_1<\infty].
\end{align*}
Furthermore $\mathbb{P}[T_2<\infty]<1$ because 
\begin{equation*}
 \left(\lambda +\frac{\lambda^2}{2}\right)\frac{e^{-i\lambda}}{1-e^{-i\lambda}}+ \frac{\lambda^2e^{-2i\lambda}}{(1-e^{-i\lambda})^2}<1.
\end{equation*}
This inequality, after same manipulation becomes
\begin{align*}
e^{-i\lambda}\left(\lambda +\frac{\lambda^2}{2}\right)+e^{-2i\lambda}\left(-\lambda +\frac{\lambda^2}{2}\right)&<\left(1-e^{-i\lambda}\right)^2\\
&=1+e^{-2\lambda i}-2e^{-i\lambda}
\shortintertext{or}
e^{-i\lambda}\left(\lambda +\frac{\lambda^2}{2}+2\right)+e^{-2i\lambda}\left(-1-\lambda +\frac{\lambda^2}{2}\right)&<1.
\shortintertext{The left-hand side of the above equation can be increased as}
e^{-i\lambda}\left[\lambda +\frac{\lambda^2}{2}+2+\frac{\lambda^2}{2}-\lambda-1\right]&=\left(1+\lambda^2\right)e^{-i\lambda}<1
\end{align*}
since $1+\lambda^2<e^{i\lambda}$.\\\\
For the iterated Poisson process $N^\alpha(N^\beta(t)),t>0$, it is possible to write the hitting probability for all $k\geq 1$ in a simple form as 
\begin{equation*}
\mathbb{P}[T_k<\infty]=e^{-\lambda_\alpha}\frac{\lambda_\alpha^k}{k!}\sum_{j=0}^{\infty}e^{-\lambda_\alpha j}\left[(j+1)^k-j^k\right], ~k\geq 1,
\end{equation*}
which coincides with equation (36) of Orsingher and Polito (2012) (see \cite{Orsingher2012})).
\bibliographystyle{abbrv}
\bibliography{researchbib}
\end{document}